\documentclass{article}

\usepackage[T1]{fontenc}
\usepackage[utf8]{inputenc}
\usepackage{amsmath, amsfonts, amsthm, amssymb, mathrsfs}

\usepackage{mathtools}
\usepackage{microtype}

\usepackage{tikz}
\usetikzlibrary{arrows.meta}

\usepackage{float}
\usepackage{enumerate}
\usepackage{graphicx}

\usepackage[hidelinks]{hyperref}

\usepackage{multicol}

 \usepackage[all]{xy}

\usepackage{paralist}
\usepackage{geometry}
\def\Z{\mathbb{Z}}
\def\R{\mathbb{R}}
\def\C{\mathbb{C}}
\def\Q{\mathbb{Q}}

\def\N{\mathbb{N}}

\def\O{\mathcal{O}}

\newtheorem{theorem}{Theorem}[section]
\newtheorem{lemma}[theorem]{Lemma}
\newtheorem{proposition}[theorem]{Proposition}
\newtheorem{corollary}[theorem]{Corollary}

\theoremstyle{definition}
\newtheorem{problem}{Problem}

\newtheorem{remark}[theorem]{Remark}
\newtheorem{example}[theorem]{Example}

\newcommand{\Int}{\operatorname{Int}}

\def\geng#1{\langle #1 \rangle}

\def\id{\operatorname{id}}

\def\deg{\operatorname{deg}}

\def\Mod{\operatorname{Mod}}
\def\Sp{\operatorname{Sp}}
\def\QU{\operatorname{QU}}

\def\rad{\operatorname{rad}}
\newcommand{\AFT}{\operatorname{AFT}}

\def\lcm{\operatorname{lcm}}

\newcommand{\comment}[1]{}

\makeatletter
\def\@addpunct#1{%
    \relax\ifhmode
    \ifnum\spacefactor>\@m \else#1\fi
    \fi}
    \newcommand{\zz}[1]{}
\newcommand{\keywordsname}{$2020$ Mathematics Subject Classification}
\def\@setkeywords{%
    {\itshape \keywordsname.}\enspace \@keywords\@addpunct.}
\def\keywords#1{\def\@keywords{#1}}
\let\@keywords=\@empty
\g@addto@macro{\maketitle}{\begingroup%
    \let\@makefnmark\relax  \let\@thefnmark\relax%
    \ifx\@keywords\@mpty\else\@footnotetext{\@setkeywords}\fi%
    \endgroup}
\makeatother

\keywords{Primary: 57K20; Secondary: 37E30, 15B36.   \\
    \indent\indent{\itshape Key words and phrases}: mapping class group, Nielsen--Thurston classification, pseudo-Anosov map, cyclotomic polynomial, symplectic representation, Morse--Smale diffeomorphism.
\\\indent\indent TK is partially supported by NSF grant DMS-2349814. {\L}PM acknowledges support from the Initiative of Excellence -- Research University (ID-UB) programme of Adam Mickiewicz University, Poznań, for a research visit to the University of Virginia, during which part of this work was developed.
}

\date{}

\title{Cyclotomic polynomials and homological criteria for mapping class types}

\author{\sc Thomas Koberda and {\L}ukasz Patryk Michalak}

\date{\today}

\numberwithin{equation}{section}

\begin{document}

    \maketitle
    \begin{abstract}
    Let $S_g$ be a closed orientable surface and let
    $\Psi\colon \Mod(S_g)\to \Sp(2g,\Z)$ be the representation induced by the action on first homology. 
    We investigate the characteristic polynomials of integral symplectic matrices arising from mapping classes of algebraically finite type and give a complete characterization in the cyclotomic case: for $n\geq 3$, the polynomial $\varphi_n(x)$ is realized by a mapping class of algebraically finite type if and only if $n$ has at most two distinct prime divisors.
    Consequently, if $n$ is square-free and has at least three distinct prime divisors, then every mapping class with characteristic polynomial $\varphi_n(x)$ is pseudo-Anosov. This gives a cyclotomic complement to the Casson--Bleiler homological criterion and yields a complete criterion for a symplectic polynomial to be realized only by pseudo-Anosov mapping classes.
    \end{abstract}

\section{Introduction}

Let $S_g$ denote the closed orientable surface of genus $g$. In the discussion below, the main case of interest is $g\geq 2$; lower-complexity surfaces enter only in auxiliary constructions. The action of the mapping class group on first integral homology gives the classical symplectic representation
\[
\Psi\colon \Mod(S_g)\longrightarrow \Sp(2g,\Z),
\]
which is surjective. Thus every integral symplectic matrix is realized as the homological action of some mapping class. In fact, for $g\geq 2$, every integral symplectic matrix is realized by a pseudo-Anosov mapping class; see, for example, \cite{Baik}. The Nielsen--Thurston classification, however, separates mapping classes into periodic, reducible and pseudo-Anosov behavior, and the homological representation does not by itself remember this decomposition. A natural problem is therefore to understand which homological data force, or obstruct, the occurrence of pseudo-Anosov components.

One well-known answer is the Casson--Bleiler homological criterion, in the symplectically irreducible form used by Margalit and Spallone: if the characteristic polynomial of $\Psi(f)$ is symplectically irreducible, is not a polynomial in $x^k$ for any $k>1$, and has no roots of unity as roots, then $f$ is pseudo-Anosov; see \cite{Casson-Bleiler,Margalit-Spallone} and Theorem 14.5 in~\cite{Farb-Marg}. This criterion has been used and sharpened in several subsequent constructions of pseudo-Anosov mapping classes with prescribed homological behavior, for instance in \cite{Margalit-Spallone}, cf.~\cite{Koberda,Malestein}. The cyclotomic case is the complementary case in which all homological eigenvalues are roots of unity, and it is orthogonal to the ``no roots of unity'' hypothesis in the Casson--Bleiler criterion.

The purpose of this paper is to analyze this cyclotomic case for mapping classes of algebraically finite type; we write $\AFT(S)$ for the set of such mapping classes on a surface $S$. Following Nielsen \cite{Nielsen1944}, a mapping class is of \emph{algebraically finite type} if it is represented by a homeomorphism which, after cutting along an invariant system of essential curves, has periodic first-return maps on the complementary components. Equivalently, its Nielsen--Thurston decomposition has no pseudo-Anosov pieces. Such mapping classes occur naturally in the study of quasi-unipotent surface homeomorphisms and of Morse--Smale isotopy classes on surfaces; see, for example, \cite{Rocha,GMM,GMMM}. By a theorem of da Rocha \cite{Rocha}, a mapping class contains a Morse--Smale diffeomorphism if and only if it is of algebraically finite type. On the other hand, the Lefschetz $\zeta$-function, together with related invariants such as Dold coefficients and algebraic periods, provides substantial information about the periodic structure of the dynamics of such diffeomorphisms, and is determined by the action induced on homology; see \cite{Dold, GMM,GMMM}. Thus, describing the image $\Psi(\AFT(S_g))$ is a central problem in the possible dynamics within the class of Morse--Smale diffeomorphisms.

Since the induced action of an algebraically finite type mapping class on homology is quasi-unipotent, one is led to the following problem:

\begin{problem}\label{problem:image_AFT}
Determine the image of the restriction
\[
\Psi|_{\AFT(S_g)}\colon \AFT(S_g)\longrightarrow \QU(2g,\Z)\cap \Sp(2g,\Z),
\]
where $\QU(2g,\Z)$ denotes the set of integral matrices all of whose eigenvalues are roots of unity.
\end{problem}

We focus on the induced characteristic polynomials $\chi_f$ coming from the actions of mapping classes $f$ on the homology of the underlying surface, instead of on the actual linear maps that occur. Our first main tool is a modern proof of Nielsen's formula for algebraically finite type mapping classes. In the periodic case, Nielsen expresses the characteristic polynomial in terms of the quotient orbifold data of the cyclic action. In the reducible algebraically finite case, the formula also includes contributions from the orbits of complementary components and from amphidrome annuli, namely annuli in the reduction system whose boundary components are exchanged by the first-return map. This will be Theorem~\ref{theorem:Nielsen_formula} below.

The principal application is a classification of cyclotomic characteristic polynomials arising from algebraically finite type mapping classes, and thus resolving Problem~\ref{problem:image_AFT} at least on the level of cyclotomic characteristic polynomials:

\begin{theorem}\label{theorem:intro_cyclotomic_classification}
Let $n\geq 3$. The cyclotomic polynomial $\varphi_n(x)$ is realizable as the characteristic polynomial of the homological action of an algebraically finite type mapping class if and only if $n$ has at most two distinct prime divisors.
\end{theorem}

The "if" direction is constructive. Prime-power and two-prime cases are obtained from periodic models realizing $\varphi_p$, $\varphi_4$ and $\varphi_{pq}$, followed by a cyclic permutation construction which replaces $x$ by a power $x^k$. The "only if" direction uses Nielsen's formula and the uniqueness of the expression of a cyclotomic polynomial as a rational product of factors of the form $x^d-1$. If $n$ has at least three distinct prime divisors, then this expression has too many numerator factors to be compatible with the Nielsen formula for algebraically finite type mapping classes.

Combining this classification with the Casson--Bleiler criterion gives a complete homological criterion for a symplectic polynomial to be realized only by pseudo-Anosov mapping classes. We prove the following:

\begin{theorem}\label{theorem:intro_complete_criterion}
	Let $p(x)\in\Z[x]$ be a symplectic polynomial of degree $2g$. Then
	$$
	\text{ $p(x)$ is realizable as $\chi_f(x)$ only for pseudo-Anosov mapping classes $f\in\Mod(S_g)$ }
	$$
	if and only if $p(x)$ satisfies each of the following conditions:
	\begin{itemize}
		\item[(a)] it is symplectically irreducible over $\Z$,
		\item[(b)] it is not a polynomial in $x^k$ for any $k > 1$;
		\item[(c)] it is not a cyclotomic polynomial, or $p(x) = \varphi_n(x)$ where $n$ has at least three distinct prime factors.
	\end{itemize}
\end{theorem}

 In particular, if $n$ is square-free and has at least three distinct prime divisors, then every mapping class $f$ satisfying $\chi_f(x)=\varphi_n(x)$ is pseudo-Anosov. The first example of a cyclotomic homological action not realized by algebraically finite type mapping classes is $n=30$, which occurs in genus $4$.

 We include some remarks on mapping class dynamics that can be recovered from virtual homological data. Specifically, Hadari~\cite{Hadari} and Liu's~\cite{Liu} resolution of a question of McMullen~\cite{McMullen} (cf.~\cite{Koberda}) shows that the presence of pseudo-Anosov dynamics is detected virtually on homology: a mapping class has a pseudo-Anosov component if and only if some lift to a finite cover has homological spectral radius greater than one. Although the homological action on the base surface $S_g$ is generally not rich enough to decide this, it is natural to ask how much dynamical information homological actions retain and, in particular, how much the homological data differs between the pseudo-Anosov and algebraically finite type cases.

The paper is organized as follows. Section~2 fixes notation. Section~3 proves Nielsen's formula in both the periodic and algebraically finite cases. Section~4 applies the formula to cyclotomic polynomials and proves Theorem~\ref{theorem:intro_cyclotomic_classification}. Section~5 combines the cyclotomic classification with the Casson--Bleiler criterion to obtain the complete criterion for pseudo-Anosov-only characteristic polynomials. Section~6 discusses finite orders of homological actions arising from algebraically finite type mapping classes.

\section{Basic notions}

Throughout the paper, mapping classes are assumed to be orientation-preserving unless explicitly stated otherwise. All surfaces are compact and orientable, and are connected unless stated to the contrary.

Let
\[
\chi_f(x):=\chi_{f_{*}}(x)=\det(xI-f_{*})
\]
be the characteristic polynomial of the induced map
$f_{*}\colon H_1(S;\Z)\to H_1(S;\Z)$. The map $f$ induces maps on homology in all dimensions, and we will write $f_{*i}$ for the induced action on $H_i$ when confusion may occur otherwise.
Since surface homology is torsion-free, we may freely pass between integral and rational homology when discussing characteristic polynomials. If $S=S_g$ is closed, then the algebraic intersection form identifies $f_{*}$ with an element of $\Sp(2g,\Z)$.

A polynomial $p(x)\in\Z[x]$ of degree $2g$ is called \emph{symplectic} if it is monic, has constant term $1$, and is reciprocal, i.e.
\[
 x^{2g}p(x^{-1})=p(x).
\]
It is \emph{symplectically reducible} if it can be written as a product of two symplectic polynomials of positive degree, and \emph{symplectically irreducible} otherwise.

A finite collection $C=\{c_i\}_{i\in I}$ of pairwise disjoint, pairwise non-isotopic essential simple closed curves in $S$ is called a \emph{reduction system} for a mapping class $f\in\Mod(S)$ if it is invariant under some representative of $f$; equivalently, after choosing such a representative, the curves in $C$ are permuted by $f$. The collection is allowed to be empty.

A mapping class $f\in\Mod(S)$ is said to be of \emph{algebraically finite type} if it has a representative, again denoted by $f$, admitting a reduction system $C$ such that, on every orbit of connected components of $S\setminus \bigcup_i c_i$, the corresponding first-return map is periodic. Equivalently, $f$ has no pseudo-Anosov component in its Nielsen--Thurston decomposition. We write $\AFT(S)$ for the set of algebraically finite type mapping classes on $S$.

	\section{Nielsen's formula for algebraically finite mapping classes}

	We now revisit results of Nielsen \cite{Nielsen1937,Nielsen1944} which yield a formula for the characteristic polynomial of a mapping class of algebraically finite type, and we give a proof in modern language. 
	
	In this section, $S$ is a connected compact orientable surface, possibly with boundary. We mostly follow Nielsen's notation and write $\omega = 1$ if $\partial S = \varnothing$ and $\omega=0$ otherwise.
	
	First, let $f \in \Mod(S)$ be periodic of order $n = n(S)$, i.e. $f^n = \id_S$. Because it is a finite order orientation-preserving homeomorphism of $S$, the quotient map $q\colon S \to S/\geng{f} =: S'$ is a branched covering with $u = u(S)$ branch points $x_1,\ldots x_{u}$. Let $h = h(S)$ be the genus of $S'$, let $b=b(S')$ be the number of boundary components of $S'$ and let $m_i = m_{i}(S) := \#q^{-1}(x_i)$ be the number of points above $x_i$. 
	
	\begin{theorem}[Nielsen 1937 \cite{Nielsen1937}] \label{theorem:Nielsen_formula_periodic}
	With the notation above, if $f\in\Mod(S)$ is periodic, then the characteristic polynomial of the induced action on $H_1(S;\Z)$ is
	\begin{align}\label{formula:Nielsen_periodic}
	\chi_{f}(x) = (x-1)^{1+\omega}
	\frac{(x^{n}-1)^{2h+u+b-2}}{(x^{m_{1}}-1)(x^{m_2}-1)\cdots(x^{m_u}-1)}.
	\end{align}
\end{theorem}

	\begin{proof}
		If $B = \{x_1,\ldots,x_u\}$ consists of all the branch points of $q$, then $q\colon S\setminus q^{-1}(B) \to S' \setminus B$ is an $n$-sheeted covering map. Take small pairwise disjoint open neighbourhoods $U_i\ni x_i$, each homeomorphic to the interior of the two-dimensional disk $\Int D^2$, such that $q^{-1}(U_i)$ is a disjoint union of $m_i$ copies of disks $\Int D_i$. Choose a finite $2$-dimensional CW-structure on the compact surface $S'' := S'\setminus \bigcup_i U_i$. After attaching the $u$ two-cells corresponding to the $U_i$, we obtain a CW-structure on $S'$ having $k_i$ cells of dimension $i$, for $i=0,1,2$. We lift the cells over $S''$ to cells of $q^{-1}(S'')=S\setminus q^{-1}(\bigcup_i U_i)$, and for each $i$ we add the $m_i$ two-cells in $q^{-1}(U_i)$. This gives a CW-structure on $S$ for which $f$ acts cellularly. It has $nk_0$ $0$-cells, $nk_1$ $1$-cells, and $n(k_2-u)+\sum_{i=1}^u m_i$ $2$-cells.
		
		Let $(C_i(S))$ be the cellular chain complex for $S$ with the CW-structure obtained above, and let $f_{\#i}\colon C_i(S)\to C_i(S)$ be the map induced by $f$. Since the characteristic polynomial is multiplicative along exact sequences of free $\Z$-modules, it is not difficult to show that
		$$
		\prod_{i=0}^2 \left( \chi_{f_{*i}}(x)\right)^{(-1)^i} = 
		\prod_{i=0}^2 \left( \chi_{f_{\#i}}(x)\right)^{(-1)^i}.
		$$
		Clearly, $\chi_{f_{*0}}(x)=(x-1)$ and $\chi_{f_{*2}}(x) = (x-1)^\omega$. Since $f$ cyclically permutes all $0$- and $1$-cells in orbits of length $n$ and $f^n = \id_S$, $\chi_{f_{\#i}}(x) = (x^n-1)^{k_i}$ for $i=0,1$. Finally, in the same way $f$ cyclically permutes $n(k_2-u)$ $2$-cells in orbits of length $n$ that lie above $S''$, and above each $U_i$ it cyclically permutes $m_i$ $2$-cells in the orbit of length $m_i$. Therefore
		$$
		\chi_{f_{\#2}}(x) = (x^n-1)^{k_2-u} \prod_{i=1}^u (x^{m_i}-1),
		$$
		Combining these identities, we obtain
		$$
		\chi_f(x) = (x-1)^{1+\omega}
		\frac{(x^{n}-1)^{k_1-k_0-k_2+u}}{(x^{m_{1}}-1)(x^{m_2}-1)\cdots(x^{m_u}-1)}
		$$
		which is the desired formula, since $k_1-k_0-k_2 = - \chi(S') = 2h + b -2$.
	\end{proof}
	
	For brevity in the sequel, we introduce the notation
	$$
	\psi_f(x) := \frac{\chi_f(x)}{(x-1)^{1+\omega}} = \frac{(x^{n}-1)^{2h+u+b-2}}{(x^{m_{1}}-1)(x^{m_2}-1)\cdots(x^{m_u}-1)}.
	$$

	Note that in the above proof we used an easy fact we will also exploit later in more generality: If $\alpha \colon \bigoplus_{i=1}^k A_i \to \bigoplus_{i=1}^k A_i$ is a homomorphism of free abelian groups $A_i$ of the same rank such that $\alpha(A_i) \subset A_{i+1}$ (where here indices are taken modulo $k$), then the characteristic polynomial of $\alpha$ satisfies
	\begin{align}\label{formula:char_pol_for_block_cyclic_matrix}
		\chi_\alpha(x) = \chi_{\alpha^k|_{A_1}}(x^k),
	\end{align}
	a formula which follows from the block cyclic form of the matrix of $\alpha$.

	Now let $f\colon S\to S$ be a homeomorphism representing a mapping class of algebraically finite type, and let $C=\{c_i\}_{i\in I}$ be a (possibly empty) reduction system for $f$ such that the first-return map on each component of $S\setminus\bigcup_i c_i$ is periodic. Choose an invariant regular neighbourhood $N(C)=\bigcup_i N(c_i)$ of $\bigcup_i c_i$, with $N(c_i)\cong c_i\times[-1,1]$, so that $f(N(C))=N(C)$.
	
	We will use the following notation:
	\begin{itemize}
		\item $\O_{S\setminus C}$ --- a set of representatives of all orbits of the action of $f$ on the set of closures of all the connected components of $S \setminus C$.
		\item $\O_C$ --- a set of representatives of all orbits of the action of $f$ on $\{c_i\}_{i\in I}$.
		\item $\ell(\Sigma)$ for $\Sigma \in \O_{S \setminus C} \cup \O_C$ --- the positive length of the orbit of $\Sigma$, i.e. the minimum positive integer such that $f^{\ell(\Sigma)}(\Sigma) = \Sigma$. 
		\item $\O^A_C$ --- the set of $c \in \O_C$ such that $f^{\ell(c)}$ swaps boundary components of $N(c)$. Adopting Nielsen's terminology, such an annulus $N(c)$ is called an \emph{amphidrome annulus}.
		\item If $\Sigma \in \O_{S\setminus C}$ or $\Sigma = N(c)$ for $c\in \O^A_C$, then the first-return mapping class $f^{\ell(\Sigma)}|_\Sigma$ is periodic of order $n(\Sigma)$; that is, it has a representative $g_\Sigma$ such that $g_\Sigma^{n(\Sigma)}=\id_\Sigma$. Following earlier definitions, $\Sigma$ naturally leads to quantities $h(\Sigma)$, $u(\Sigma)$, $b(\Sigma)$ and $m_i(\Sigma)$ for $i=1,\ldots,u(\Sigma)$.
	\end{itemize}
	
	\begin{theorem}[Nielsen 1944 \cite{Nielsen1944}] \label{theorem:Nielsen_formula}
	With the notation above, if $f\in\Mod(S)$ is of algebraically finite type, then the characteristic polynomial of the induced action on $H_1(S;\Z)$ is
		\begin{align} \label{formula:Nielsen_formula_branch_datum}
			\chi_f(x)
			&=(x-1)^{1+\omega}
			\prod_{\Sigma\in\mathcal O_{S\setminus C} \cup \O^A_C}
			\frac{\bigl(x^{\,\ell(\Sigma)\,n(\Sigma)}-1\bigr)^{\,2h(\Sigma)+u(\Sigma)+b(\Sigma)-2}}
			{\displaystyle\prod_{i\in\{1,\ldots,u(\Sigma)\}}
				\bigl(x^{\,\ell(\Sigma)\,m_i(\Sigma)}-1\bigr)}
			\\\label{formula:Nielsen_formula_psi}
			&=(x-1)^{1+\omega}
			\prod_{\Sigma\in\mathcal O_{S\setminus C}}
			\psi_{(f^{\,\ell(\Sigma)})|_\Sigma}\!\bigl(x^{\,\ell(\Sigma)}\bigr)
			\prod_{c\in\mathcal O_C^{A}}
			\frac{x^{\ell(c)}+1}{x^{\ell(c)}-1}.
		\end{align}

	\end{theorem}

	 Note that the product is taken over orbits of closures of components of $S \setminus C$ and amphidrome annuli. Non-amphidrome annuli make no contribution to $\chi_f(x)$. 
	 
	 Before we proceed to the proof of this formula, recall that the mapping class group $\Mod(A)$ of mappings of an annulus $A = S^1 \times [-1,1]$ that can swap its boundary components is of the form
	 $$
	 \Mod(A) = \geng{T,R} \cong \Z \rtimes \Z/2\Z \cong D_\infty,
	 $$
	 where $D_\infty$ is the infinite dihedral group, $T$ is the Dehn twist about $S^1 \times \{0\}$, and $R$ is an orientation-preserving involution of $A$ exchanging its boundary components. More specifically, one can put
	 $$
	 T(\theta,t) = (\theta+\pi(t+1),t) \ \text{ and } \ R(\theta,t) = (-\theta,-t) \ \text{ for } \ (\theta,t)\in S^1\times[-1,1],
	 $$
	 which makes it easy to check that $RTR^{-1} = T^{-1}$, $R^2 = \id_A$. Moreover, straightforward computation shows
	 $$
	 \chi_{T^k}(x) = x-1 \ \ \text{ and } \ \ \chi_{RT^k}(x) = x+1.
	 $$
	 Finally, if $N(c)$ is an amphidrome annulus, then $g_{N(c)} = RT^k$ for some $k$, so it is of order $n = 2$ with two fixed points $(0,0)$ and $(\pi,0)$, thus $u=2$ and $m_1 = m_2 =1$. The quotient $N(c)/\geng{RT^k}$ is homeomorphic to a disc, so $h = 0$ and $b=1$. Therefore
	 $$
	 \psi_{RT^k}(x) = \frac{x+1}{x-1} = \frac{(x^2-1)^{2h+u+b-2}}{(x-1)(x-1)}
	 $$
	 is consistent with formula~\eqref{formula:Nielsen_periodic} and justifies the equality of formulas~\eqref{formula:Nielsen_formula_branch_datum} and~\eqref{formula:Nielsen_formula_psi}.

	 \begin{proof}[Proof of Theorem \ref{theorem:Nielsen_formula}]
	 	Take $U = S \setminus C$ and $V = N(C)$, then their interiors cover $S$ and $U\cap V \simeq C_- \sqcup C_+$, where $C_\pm = C \times \{\pm 1\} \subset N(C)$. Consider the Mayer--Vietoris long exact sequence for $S=U\cup V$:
\[
\begin{gathered}
0\to H_2(S)\to H_1(U\cap V)\to H_1(U)\oplus H_1(V)\to H_1(S)\\
\to H_0(U\cap V)\to H_0(U)\oplus H_0(V)\to H_0(S)\to 0.
\end{gathered}
\]
	 	Since $f(U) = U$ and $f(V) = V$, from the naturality of the sequence and multiplicativity of the characteristic polynomial for exact sequences of free $\Z$-modules, $\chi_f(x) = \det(x I - f_{*}\vert_{H_1(S)})$ is equal to
		\[
		\begin{aligned}
		\chi_f(x)=&\,\det(x I-f_{*}\vert_{H_0(S)})\,
		\det(x I-f_{*}\vert_{H_2(S)})\\
		&\cdot
		\underbrace{\frac{\det(x I-f_{*}\vert_{H_1(U)})}{\det(x I-f_{*}\vert_{H_0(U)})}}_{=:p_U(x)}\\
		&\cdot
		\underbrace{\frac{\det(x I-f_{*}\vert_{H_1(V)})\,
		\det(x I-f_{*}\vert_{H_0(U\cap V)})}
		{\det(x I-f_{*}\vert_{H_1(U\cap V)})\,
		\det(x I-f_{*}\vert_{H_0(V)})}}_{=:p_V(x)}.
		\end{aligned}
		\]
		 Clearly, $\det(x I - f_{*}\vert_{H_0(S)})\, \det(x I - f_{*}\vert_{H_2(S)}) = (x-1)^{1+\omega}$, because $S$ is connected and $f$ is orientation-preserving.
		 
		 For $p_U(x)$, we use the splitting of the homology groups along orbits and formula~\eqref{formula:char_pol_for_block_cyclic_matrix}:
		 \begin{align*}
		p_U(x)
		 = 
		 \prod_{\Sigma \in \O_{S\setminus C}}
	 	\frac{
	 		\det(x^{\ell(\Sigma)} I - f^{\ell(\Sigma)}_{*}\vert_{H_1(\Sigma)})
	 	}{
	 		\det(x^{\ell(\Sigma)} I - f^{\ell(\Sigma)}_{*}\vert_{H_0(\Sigma)})
	 	}
	 	=
	 	\prod_{\Sigma \in \O_{S\setminus C}}
	 	\frac{
	 		\chi_{f^{\ell(\Sigma)}|_\Sigma}(x^{\ell(\Sigma)})
	 	}{
	 		(x^{\ell(\Sigma)}-1)
	 	}
	 	=
	 	\prod_{\Sigma \in \O_{S\setminus C}}
	 	\psi_{f^{\ell(\Sigma)}|_\Sigma}(x^{\ell(\Sigma)}).
		 \end{align*}
		Finally, for $p_V(x)$ further identify $H_*(V) = H_*(C)$ and $H_*(U\cap V) = H_*(C_- \sqcup C_+)$ and use the fact that $(f^{\ell(c)}|_c)^2 = \id_c$ for $c \in \O^A_C$:
		 \begin{align*}
		 p_V(x)
		 &= 
		 \prod_{c \in \O_{C}}
		 	\frac{
				\chi_{f^{\ell(c)}|_c}(x^{\ell(c)})
				\det(x^{\ell(c)} I - (f^{\ell(c)})_{*}\vert_{H_0(c_-\sqcup c_+)})
			}{
				\det(x^{\ell(c)} I - (f^{\ell(c)})_{*}\vert_{H_1(c_-\sqcup c_+)})\,
				\det(x^{\ell(c)} I - (f^{\ell(c)})_{*}\vert_{H_0(c)})
			}
		 \\&=
		 \prod_{c \in \O_{C} \setminus \O_{C}^A}
		 \frac{
		 	(x^{\ell(c)}-1)\,
		 	(x^{\ell(c)}-1)^2
		 }{
		 	(x^{\ell(c)}-1)^2\,
		 	(x^{\ell(c)}-1)
		 }\,
		 \prod_{c \in  \O_{C}^A}
		 \frac{
		 (x^{\ell(c)}+1)\,
		 (x^{2\ell(c)}-1)
		 }{
		 (x^{2\ell(c)}-1)\,
		 (x^{\ell(c)}-1)
		 }
		 \\&=
		\prod_{c\in\mathcal O_C^{A}}
		\frac{x^{\ell(c)}+1}{x^{\ell(c)}-1}.
		 \end{align*}
	This proves the desired formulas.
	 \end{proof}

\section{Realizability criteria for cyclotomic polynomials}

\subsection{Reduction systems and irreducible characteristic polynomials}
	The following standard observation will be used repeatedly.

    \begin{lemma}\label{lemma:irreducible_polynomial--only_separating_curves}
Let \(g\geq 2\), and let \(f\in \operatorname{Mod}(S_g)\) be reducible.
Suppose that \(\chi_f(x)\) is symplectically irreducible over \(\mathbb Z\).
Then no reduction system for \(f\) contains a non-separating curve.
\end{lemma}

\begin{proof}
Let \(A=f_*\) denote the induced symplectic automorphism of
\(V=H_1(S_g;\mathbb Q)\). Suppose, for a contradiction, that some
\(f\)-invariant reduction system \(C\) contains a non-separating curve
\(c\). Let \(\mathcal O=\{c_1,\ldots,c_m\}\) be the \(f\)-orbit of \(c\),
and set
\[
  W=\langle [c_1],\ldots,[c_m]\rangle_{\mathbb Q}
  \subset H_1(S_g;\mathbb Q).
\]
Since \(c\) is non-separating, \(W\neq 0\). Since \(A\) permutes the
homology classes \([c_i]\) up to sign, \(W\) is \(A\)-invariant, and
\(A|_W\) has finite order. In particular, $q(x):=\chi_{A|_W}(x)$ is a product of cyclotomic polynomials.

The curves in \(\mathcal O\) are pairwise disjoint, so the algebraic
intersection form vanishes on \(W\). Thus \(W\) is isotropic. Hence
\[
  0\subset W\subset W^\perp\subset V
\]
is an \(A\)-invariant filtration, and \(W^\perp/W\) inherits a
nondegenerate symplectic form. Let \(\bar A\) be the induced symplectic
automorphism of \(W^\perp/W\), and write $\bar q(x) :=\chi_{\bar A}(x)$. The action induced by \(A\) on \(V/W^\perp\) is dual to \(A|_W^{-1}\).
Since \(A|_W\) has finite order, \(q(x)\) is a product of cyclotomic
polynomials, and hence its multiset of roots is invariant under
inversion. Therefore
the characteristic polynomial of the action on \(V/W^\perp\) is again
\(q(x)\). It follows from the above filtration that
\[
  \chi_A(x)=q(x)^2\bar q(x).
\]

Now \(q(x)^2\) is a symplectic polynomial; indeed, \(q(x)\) is a product of
cyclotomic polynomials, and squaring accounts for the possible factors
\(\varphi_1(x)=x-1\) and \(\varphi_2(x)=x+1\). Also, \(\bar q(x)\) is
symplectic, because it is the characteristic polynomial of a symplectic
automorphism of \(W^\perp/W\).

If \(W\neq W^\perp\), then both \(q(x)^2\) and \(\bar q(x)\) have
positive degree, so $\chi_A(x)=q(x)^2\bar q(x)$ is a nontrivial factorization into symplectic polynomials. This contradicts the symplectic irreducibility of \(\chi_f(x)\).

It remains to consider the case \(W=W^\perp\), so that \(W\) is
Lagrangian. Then \(\bar q(x)=1\) and $\chi_A(x)=q(x)^2$. Since \(g\geq 2\), the polynomial \(q(x)^2\) is still symplectically reducible. Indeed, writing
\[
  q(x)=\prod_m \varphi_m(x)^{e_m},
\]
we note that the $m$th cyclotomic polynomial \(\varphi_m(x)\) is symplectic for \(m\geq 3\) (see, for instance, \cite{Margalit-Spallone,Yang}), while \(\varphi_1(x)=x-1\) and \(\varphi_2(x)=x+1\) become symplectic only in pairs:
\[
  \varphi_1(x)^2=(x-1)^2
  \quad\text{and}\quad
  \varphi_2(x)^2=(x+1)^2.
\]
Since \(\deg q=g\geq 2\), the factors of \(q(x)^2\) can
be divided into two nonempty products of such symplectic factors. Hence
\(\chi_A(x)=q(x)^2\) is symplectically reducible, again contradicting
the hypothesis.

Therefore no reduction system for \(f\) contains a non-separating curve.
\end{proof}
	
	
	The next lemma isolates the part of the Casson--Bleiler argument \cite[Theorem 14.5]{Farb-Marg} which will be needed below.

\begin{lemma}\label{lemma:orbits_for_reducible_mappings_with_irreducible_char_polynomial}
Let \(f \in \operatorname{Mod}(S_g)\) be reducible and suppose that
\(\chi_f(x)\) is symplectically irreducible over \(\mathbb Z\). Then
there is a reduction system \(C=\{c_1,\ldots,c_k\}\) such that exactly
one of the following holds:
\begin{itemize}
\item \(k=1\), the two components of \(S_g\setminus C\) are exchanged by
\(f\), and \(N(c_1)\) is an amphidrome annulus; equivalently,
\(O_{S\setminus C}=\{\Sigma\}\) and \(\ell(\Sigma)=2\);
\item \(k\geq 3\), the components of \(S_g\setminus C\) consist of an
\(f\)-invariant sphere \(S_{0,k}\) with \(k\) boundary components and
\(k\) homeomorphic positive-genus components cyclically permuted by
\(f\). Equivalently, \(O_{S\setminus C}=\{\Sigma,S_{0,k}\}\),
\(\ell(\Sigma)=k\), \(\ell(S_{0,k})=1\), and \(O_C^A=\varnothing\).
\end{itemize}
\end{lemma}

\begin{proof}
By Lemma~\ref{lemma:irreducible_polynomial--only_separating_curves}, every curve in every reduction system for $f$ is separating. Choose a nonempty reduction system $C$ which is minimal with respect to inclusion among nonempty $f$-invariant multicurves. Equivalently, $C$ is a single $f$-orbit of separating curves. Let $\Gamma_C$ be the dual graph of the decomposition of $S_g$ along $C$; vertices correspond to components of $S_g\setminus C$, and edges correspond to curves of $C$. Since the curves of $C$ are separating, $\Gamma_C$ is a tree, and $f$ induces an automorphism of this tree.

We first record the symplectic consequence of the hypothesis. For a
component \(Y\) of \(S_g\setminus C\), let \(G(Y)\) denote the image of
\(H_1(Y;\mathbb Q)\) in \(H_1(S_g;\mathbb Q)\), after quotienting out
the boundary classes. Equivalently, \(G(Y)\) is the genus part of the
homology of \(Y\). Since all curves of \(C\) are separating, each
\(G(Y)\) is a symplectic subspace of dimension \(2\operatorname{genus}(Y)\),
and there is an orthogonal symplectic direct-sum decomposition
\[
  H_1(S_g;\mathbb Q)
  \cong
  \bigoplus_{\operatorname{genus}(Y)>0} G(Y).
\]
The map \(f_*\) permutes the nonzero summands in this decomposition.

If the positive-genus components formed more than one \(f\)-orbit, then
the direct sum of the summands \(G(Y)\) over one such orbit would be a
proper nonzero \(f_*\)-invariant symplectic subspace of
\(H_1(S_g;\mathbb Q)\). Moreover, this subspace is defined by an
integral symplectic sublattice coming from the corresponding union of
subsurfaces. Hence \(\chi_f(x)\) would split as a product of two
positive-degree symplectic polynomials, contradicting the symplectic
irreducibility of \(\chi_f(x)\). Therefore the positive-genus vertices
of \(\Gamma_C\) form a single \(f\)-orbit.

Now consider the leaves of $\Gamma_C$. Every leaf has positive genus: a genus-zero leaf would be a sphere with one boundary component, and its boundary curve would be inessential in $S_g$. Since the positive-genus vertices form one orbit, and since the set of leaves is invariant under every automorphism of a tree, it follows that the leaves of $\Gamma_C$ form a single $f$-orbit. Moreover, every positive-genus vertex is a leaf, because there is at least one leaf and the orbit of any positive-genus leaf is the whole set of positive-genus vertices.

We will also use that no edge of $\Gamma_C$ joins two genus-zero vertices. Indeed, $C$ is a single $f$-orbit of curves, so if one edge joined two genus-zero vertices, then every edge would join two genus-zero vertices. This is impossible because every tree has a leaf, and every leaf of $\Gamma_C$ has positive genus.

The automorphism of a finite tree induced by $f$ has either a fixed vertex or a fixed edge. Suppose first that it has a fixed edge $e$. The edge $e$ cannot join two genus-zero vertices, by the preceding paragraph. It also cannot join a positive-genus vertex to a genus-zero vertex: in that case the two endpoints would be fixed individually, giving a fixed positive-genus vertex, whereas the positive-genus vertices form a single orbit and include at least two leaves. Hence both endpoints of $e$ have positive genus. Since all positive-genus vertices form one orbit, there are exactly two such vertices; since every positive-genus vertex is a leaf, the tree consists of the single edge $e$. Thus $C$ consists of one separating curve, and $f$ exchanges the two complementary components of $S_g\setminus C$. Equivalently, the annular neighbourhood of this curve is amphidrome annulus. This is the first case.

It remains to treat the case in which the induced automorphism fixes a vertex $v$. The vertex $v$ cannot have positive genus, for then the single-orbit property for positive-genus vertices would imply that there is only one positive-genus vertex, contradicting the existence of at least two positive-genus leaves. Thus $v$ has genus zero.

All positive-genus vertices are leaves and lie in one orbit, so they all have the same distance from $v$. This common distance is equal to one. Indeed, if a path from $v$ to a positive-genus leaf had length greater than one, then the first edge on that path would join $v$ to a genus-zero vertex, because any positive-genus vertex is already a leaf; this would contradict the fact that no edge joins two genus-zero vertices. Hence $\Gamma_C$ is a star: it has a single central genus-zero vertex $v$, and all other vertices are positive-genus leaves adjacent to $v$.

The component corresponding to $v$ is therefore a sphere with $k$ boundary components, where $k$ is the number of edges of $\Gamma_C$, equivalently the number of curves in $C$. The $k$ positive-genus components are permuted transitively, hence cyclically, by $f$. The cases $k=1$ and $k=2$ are excluded: for $k=1$ the central sphere would be a disk, making the unique curve in $C$ inessential, while for $k=2$ the central sphere would be an annulus, making the two curves of $C$ isotopic. Thus $k\geq 3$. Finally, after the first return to any curve of $C$, both the central sphere and the adjacent positive-genus component are preserved setwise. Therefore the two sides of the corresponding annular neighborhood are not exchanged, so no annulus is an amphidrome annulus. This is the second case.
\end{proof}

\subsection{Numerator indices}

 	 We denote the $n^{th}$ cyclotomic polynomial by $\varphi_n(x)$. It is well-known that
\begin{equation}\label{formula:cyclotomic_polynomial}
	\varphi_n(x)=\prod_{d\mid n}(x^d-1)^{\mu(n/d)},
\end{equation}
where $\mu \colon \N \to \{-1,0,1\}$ is the M{\"o}bius function.

We will need the following elementary bookkeeping fact. Suppose that a nonzero polynomial $P(x)\in\Q[x]$ is written as a \emph{finite rational product} of the form
\[
P(x)=\prod_{d\geq 1}(x^d-1)^{a_d}, \qquad a_d\in\Z.
\]
Then the exponents $a_d$ are uniquely determined by the polynomial $P$. Indeed, after factoring $x^d-1=\prod_{e\mid d}\varphi_e(x)$, unique factorization in $\Q[x]$ determines the exponent $b_e$ of each cyclotomic factor $\varphi_e(x)$ in $P(x)$, and these exponents satisfy
\[
b_e=\sum_{e\mid d} a_d.
\]
Starting with the largest $d$ for which $a_d\neq 0$ and descending, these equations determine all $a_d$. Thus it makes sense to speak of the \emph{numerator indices} of such an expression, namely the set of $d$ for which $a_d>0$. Formula~\eqref{formula:cyclotomic_polynomial} is the corresponding unique expression for $\varphi_n(x)$.

The next lemma shows that Nielsen's formula imposes a strong restriction on symplectically irreducible characteristic polynomials coming from algebraically finite type mapping classes. Since all roots of such polynomials are roots of unity, they are products of cyclotomic polynomials; in the symplectically irreducible case the only possibilities are $\varphi_n(x)$ for $n\geq 3$, together with the exceptional cases $(x-1)^2$ and $(x+1)^2$.

\begin{lemma}\label{lemma:aft_irreducible_at_most_three_numerators}
Let $f\in\AFT(S_g)$ and suppose that $\chi_f(x)$ is symplectically irreducible over $\Z$. Write
\[
\chi_f(x)=\prod_{d\geq 1}(x^d-1)^{a_d}
\]
in the unique rational product form above. Then there are at most three indices $d$ for which $a_d>0$, i.e.~at most three numerator indices.
\end{lemma}

\begin{proof}
If $f$ is periodic, then Nielsen's periodic formula~\eqref{formula:Nielsen_periodic} shows that the only possible numerator indices are $1$ and the order of $f$. Thus there are at most two numerator indices.

Suppose now that $f$ is reducible. Since $\chi_f(x)$ is a symplectically irreducible polynomial, Lemma~\ref{lemma:orbits_for_reducible_mappings_with_irreducible_char_polynomial} applies. Choose the reduction system described there and use Nielsen's formula~\eqref{formula:Nielsen_formula_branch_datum}. In the first case of Lemma~\ref{lemma:orbits_for_reducible_mappings_with_irreducible_char_polynomial}, there is one orbit of positive-genus complementary components, of length $2$, and one amphidrome annulus. The possible numerator indices are therefore contained in
\[
\{1,\;2n(\Sigma),\;2\},
\]
where $n(\Sigma)$ is the order of the first-return map on the positive-genus component $\Sigma$. Here the index $1$ comes from the global factor $(x-1)^2$, the index $2n(\Sigma)$ from the periodic first-return map on the component orbit, and the index $2$ from the amphidrome annulus.

In the second case, the complementary components consist of one $f$-invariant sphere $S_{0,k}$ and one orbit of $k\geq 3$ positive-genus components, with no amphidrome annuli. If $n(\Sigma)$ denotes the order of the first-return map on a positive-genus component and $n_0$ denotes the order of the first-return map on the central sphere, then the possible numerator indices are contained in
\[
\{1,\;k n(\Sigma),\; n_0\}.
\]
Again the index $1$ is the global factor, while the other two indices come from the two component orbits in Nielsen's formula. Denominator factors only subtract from the exponents $a_d$ and cannot create new positive indices. Hence in all cases there are at most three numerator indices.
\end{proof}

\subsection{Realizability}

We want to determine which cyclotomic polynomials are realizable as $\chi_f(x)$ for mapping classes $f\in\Mod(S)$ of algebraically finite type.
If $S$ has genus $g$, then for any $d$ dividing $4g$ or $4g+2$ one can easily provide a periodic mapping of $S$ of order $d$ by applying a rotation through $2\pi/d$ to the regular $(4g)$- or $(4g+2)$-gon representing $S$. In particular, if $p$ is an odd prime and $g=\frac{p-1}{2}$, then $d=p$ and $d=2p$ yield realizations of $\varphi_d(x)$ as $\chi_f(x)$ for a periodic mapping $f\in \Mod(S)$.

\begin{example}\label{example:periodic_of_order_pq}
	Let $p\neq q$ be distinct primes and $n=pq$. Consider the affine plane curve in $\C^2$ defined by
	$$
	y^{pq}=x^{p}(1-x)^{q}.
	$$
	Its projective closure has a smooth normalization, a connected closed Riemann surface $S$ of genus $g=\frac{(p-1)(q-1)}{2}$, such that the projection $\pi\colon S \to \C\mathrm{P}^{1}$ given by $(x,y)\mapsto x$ is a cyclic branched covering of degree $pq$. It has branch points $0$, $1$ and $\infty$: above $0$ there are $m_0 = p$ points each of ramification index $q$, above $1$ there are $m_1=q$ points each of ramification index $p$, and above $\infty$ there is a single point ($m_\infty=1$) of ramification index $pq$.
	
	Define a periodic mapping $f \in \Mod(S)$ of order $pq$ given by $f(x,y) = (x,\zeta_{pq}y)$, where $\zeta_{pq}$ is a primitive root of unity of order $pq$. Nielsen's formula~\eqref{formula:Nielsen_periodic} for $f$ yields
	$$
	\chi_f(x) = (x-1)^2\frac{ (x^{pq}-1) }{(x^p-1)(x^q-1)(x-1)} = \frac{ (x^{pq}-1)(x-1) }{(x^p-1)(x^q-1)} = \varphi_{pq}(x).
	$$
\end{example}

	\begin{example}\label{example:genus_2}
		By a direct Riemann--Hurwitz computation, $\Mod(S_2)$ has periodic elements of exactly the orders $\{1,2,3,4,5,6,8,10\}$; see \cite{Harvey}. Among them, $\varphi_n(x)$ occurs for $n\in\{5,8,10\}$. The only remaining possibility for a finite order of an element in $Sp(4,\Z)$ is $12$, that can occur for matrices with the characteristic polynomials $\varphi_{12}(x)$ or $\varphi_3(x)\cdot\varphi_4(x)$. The latter case is easily obtained in the class of algebraically finite type by taking periodic mappings of tori of order $3$ and $4$ and performing their connected sum along invariant neighbourhoods of fixed points.
		
		It turns out that $\varphi_{12}$ is also realized as $\chi_f(x)$ for $f$ of algebraically finite type. 
		The maximum order of a torsion element in $\Mod(S_2)$ is $10$, so in both the above cases the mappings are not periodic.

		Cut $S_2$ along an essential separating curve $c$, as in Figure \ref{figure:s2_with_separating_curve}, obtaining two components $\Sigma$ and $\Sigma'$. For $f\in \Mod(S_2)$ put $g$ on $\Sigma$ satisfying $\chi_g(x) = \varphi_6(x)$ and the identity on $\Sigma'$ and then compose with an order $2$ rotation permuting $\Sigma$ and $\Sigma'$. Then $f$ is orientation-preserving and the regular neighbourhood $N(c)$ of $c$ is an amphidrome annulus. Therefore Nielsen's formula~\eqref{formula:Nielsen_formula_psi} gives
		$$
		\chi_f(x) = (x-1)^2\cdot  \frac{\varphi_6(x^2)}{x^2-1} \cdot \frac{x+1}{x-1} = \varphi_6(x^2) = \varphi_{12}(x).
		$$
	\end{example}

	The polynomial $\varphi_{12}(x) = x^4 - x^2 + 1$ is the characteristic polynomial of the following symplectic matrix:
	$$
	\begin{bmatrix}
		0 & 1 & 0 & 0 \cr
		1 & 0 & -1 & 0 \cr
		0 & 0 & 0 & 1 \cr
		1 & 0 & 0 & 0
	\end{bmatrix}.
	$$

\begin{corollary}
	The following are characteristic polynomials of symplectic transformations induced by mapping classes of algebraically finite type on $S_2$:
	\begin{align*}
		&f(x)g(x), \text{ where } f(x),g(x) \in \{ (x-1)^2, (x+1)^2, \varphi_3(x), \varphi_4(x), \varphi_6(x) \}, \\
		&\varphi_5(x),\, \varphi_8(x),\, \varphi_{10}(x),\,\textrm{ and } \varphi_{12}(x).
	\end{align*}
\end{corollary}

In fact, $\varphi_{12}(x)$ has a somewhat unusual property: it is realized by a reducible mapping class, but only by one without any pseudo-Anosov pieces.

	\begin{proposition}\label{proposition:phi12_no_mixed_reducible}
	The polynomial $\varphi_{12}(x)$ is realized as $\chi_f(x)$ by a non-periodic algebraically finite type mapping class $f\in\Mod(S_2)$, but it is not realized by any reducible mapping class with a pseudo-Anosov component.
	\end{proposition}

	\begin{proof}
		The construction of a non-periodic algebraically finite type mapping class $f\in\Mod(S_2)$ with $\chi_f(x)=\varphi_{12}(x)$ is given in Example~\ref{example:genus_2}.
				
		Suppose that $f$ is reducible with a pseudo-Anosov component and $\chi_f(x)=\varphi_{12}(x)$. Since $\varphi_{12}(x)$ is irreducible, Lemma~\ref{lemma:irreducible_polynomial--only_separating_curves} implies that every curve in any reduction system for $f$ is separating. The maximal number of pairwise disjoint, pairwise non-isotopic separating simple closed curves in $S_g$ is $2g-3$, so in genus $2$ the reduction system consists of a single curve $c$, as in Figure~\ref{figure:s2_with_separating_curve}.

        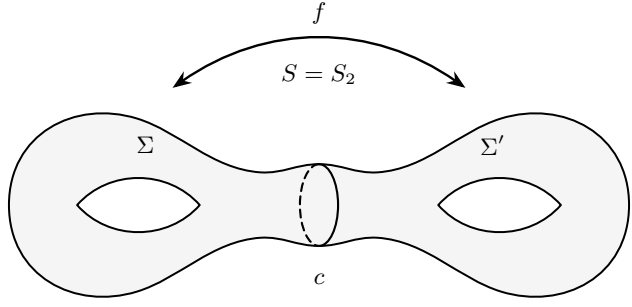
\begin{figure}[t]
\centering
\begin{tikzpicture}[
  scale=0.9,
  line cap=round,
  line join=round,
  >=Stealth,
  every node/.style={font=\small}
]
  
  \path[draw=black, fill=gray!8, line width=0.75pt]
    (-4.55,0)
      .. controls (-4.55,1.05) and (-3.58,1.55) .. (-2.67,1.27)
      .. controls (-2.05,1.08) and (-1.58,0.53) .. (-0.86,0.48)
      .. controls (-0.52,0.46) and (-0.34,0.60) .. (0,0.60)
      .. controls (0.34,0.60) and (0.52,0.46) .. (0.86,0.48)
      .. controls (1.58,0.53) and (2.05,1.08) .. (2.67,1.27)
      .. controls (3.58,1.55) and (4.55,1.05) .. (4.55,0)
      .. controls (4.55,-1.05) and (3.58,-1.55) .. (2.67,-1.27)
      .. controls (2.05,-1.08) and (1.58,-0.53) .. (0.86,-0.48)
      .. controls (0.52,-0.46) and (0.34,-0.60) .. (0,-0.60)
      .. controls (-0.34,-0.60) and (-0.52,-0.46) .. (-0.86,-0.48)
      .. controls (-1.58,-0.53) and (-2.05,-1.08) .. (-2.67,-1.27)
      .. controls (-3.58,-1.55) and (-4.55,-1.05) .. (-4.55,0)
    -- cycle;

  \path[draw=black, fill=white, line width=0.75pt]
    (-3.55,0)
      .. controls (-3.08,0.53) and (-2.22,0.53) .. (-1.75,0)
      .. controls (-2.22,-0.53) and (-3.08,-0.53) .. (-3.55,0)
    -- cycle;

  \path[draw=black, fill=white, line width=0.75pt]
    (1.75,0)
      .. controls (2.22,0.53) and (3.08,0.53) .. (3.55,0)
      .. controls (3.08,-0.53) and (2.22,-0.53) .. (1.75,0)
    -- cycle;

	\draw[densely dashed,line width=0.8pt]
	(0,0.60) arc[start angle=90,end angle=270,x radius=0.28,y radius=0.60];
	\draw[solid,line width=0.8pt] 
	(0,-0.60) arc[start angle=-90,end angle=90,x radius=0.28,y radius=0.60];

  \node at (-2.55,0.88) {$\Sigma$};
  \node at ( 2.55,0.88) {$\Sigma'$};
  \node at (0,-1.08) {$c$};
  \node at (0,1.92) {$S=S_2$};

  \draw[<->, line width=0.9pt]
    (-2.15,1.72) to[out=35,in=145] node[above=2pt] {$f$} (2.15,1.72);
\end{tikzpicture}
\caption{A genus two surface cut along a separating curve \(c\), with \(f\) exchanging the two complementary components.}
\label{figure:s2_with_separating_curve}
\end{figure}

		The mapping class $f$ must exchange the two components of $S_2\setminus c$; otherwise the homology of each component would give a proper invariant subspace of $H_1(S_2;\Q)$, contradicting the irreducibility of $\chi_f(x)$. 
        Let $\Sigma$ and $\Sigma'$ be the closures of these components; $f^2$ preserves each of them separately. By assumption, and after possibly interchanging $\Sigma$ and $\Sigma'$, we may suppose $f^2$ is pseudo-Anosov on $\Sigma=S_{1,1}$.
        The homological action of $\Mod(S_{1,1})$ on $H_1(S_{1,1};\Z)$ lies in $SL(2,\Z)$, and a pseudo-Anosov first-return map induces a hyperbolic element. Thus $(f^2|_\Sigma)_*$ is conjugate over $\R$ to
		$$
		A_\lambda = \begin{bmatrix}
			\lambda & 0 \cr
			0 & \lambda^{-1}
		\end{bmatrix},
		$$
		where $\lambda\in\R$ and $|\lambda|>1$. Hence $f^2_{*}$ is conjugate to $A_\lambda\oplus B$, for a suitable real number $\lambda\neq\pm 1$ and a matrix $B\in SL(2,\Z)$
        induced by $f^2|_{\Sigma'}$. 
		
		Note that
		$$
		f^2|_{\Sigma'} = f|_{\Sigma} \circ f^2|_{\Sigma} \circ (f|_{\Sigma})^{-1},
		$$
		where $f|_{\Sigma} \colon \Sigma \longrightarrow \Sigma'$, so $f^2|_{\Sigma}$ and $f^2|_{\Sigma'}$ are conjugate, and so they are both pseudo-Anosov. In particular, their homological actions $A_\lambda$ and $B$ are conjugate.

        Therefore, $f^2_{*}$ has four real eigenvalues which are distinct from $\pm 1$. On the other hand, if $\zeta_{12}$ is a primitive root of unity of order $12$, then the squares of the roots of $\varphi_{12}$ are non-real: $\zeta_{12}^2=\zeta_6$, $(\zeta_{12}^5)^2=\zeta_6^5$, $(\zeta_{12}^7)^2=\zeta_6$, and $(\zeta_{12}^{11})^2=\zeta_6^5$. This is a contradiction.
		\end{proof}

\begin{lemma}\label{lemma:reducible_mapping_and_irreducible_polynomial_r(x^k)}
	Let $f\in\Mod(S_g)$ be such that $\chi_f(x)$ is symplectically irreducible over $\Z$. If $f$ is reducible, then $\chi_f(x)=r(x^k)$ for some symplectic polynomial $r(x)$ and some $k>1$.
\end{lemma}

\begin{proof}
Choose a reduction system as in Lemma~\ref{lemma:orbits_for_reducible_mappings_with_irreducible_char_polynomial}. In the first case, $f$ exchanges the two positive-genus components cut off by the unique separating curve. The first return map to either component is $f^2|_\Sigma$, and the block-cyclic form of the action on
$H_1(S_g;\Q)\cong H_1(\Sigma;\Q)\oplus H_1(f(\Sigma);\Q)$ gives
\[
\chi_f(x)=\chi_{(f^2|_\Sigma)_{*}}(x^2).
\]
Thus $\chi_f(x)=r(x^2)$.

In the second case, $f$ cyclically permutes $k\geq3$ homeomorphic positive-genus components, while the central sphere contributes no homology to the direct-sum decomposition of $H_1(S_g;\Q)$. Applying the same block-cyclic computation to the $k$ positive-genus summands gives
\[
\chi_f(x)=\chi_{(f^k|_\Sigma)_{*}}(x^k).
\]
In both cases the polynomial $r$ is the characteristic polynomial of the homological action of a mapping class on a compact surface with one boundary component, and is therefore symplectic.
\end{proof}

Recall that
$$
\varphi_n(x) = \varphi_{\rad(n)}(x^{n/\rad(n)}),
$$ 
where here $\rad(n)$ is the \emph{radical} of $n$, the product of all its distinct prime divisors. One can check that $\varphi_n(x)$ is not of the form $r(x^k)$ for some $k>1$ if and only if $n$ is square-free, equivalently $n=\rad(n)$.

The following elementary construction is the topological source of the substitutions $x\mapsto x^m$ which appear below.

\begin{lemma}[Cyclic substitution construction]\label{lemma:cyclic_substitution_construction}
Let $\Sigma$ be a compact orientable surface with one boundary component, and let $h$ be a mapping class of $\Sigma$. Put $q(x)=\chi_h(x)$. For every $m\geq 1$, there is a mapping class $F$ of the closed surface of genus $m\cdot\operatorname{genus}(\Sigma)$ such that
\[
\chi_F(x)=q(x^m).
\]
If $m\geq 2$, the mapping class $F$ is reducible. If, moreover, $h$ is periodic of order $n$, then $F$ may be chosen to be of algebraically finite type such that $F_*$ has finite order $mn$.
\end{lemma}

\begin{proof}
Let $\Sigma_1,\ldots,\Sigma_m$ be $m$ copies of $\Sigma$, and attach their boundary components to the boundary components of a genus-zero surface $S_{0,m}$. The resulting closed surface $S$ has genus $m\cdot\operatorname{genus}(\Sigma)$. Choose the gluing data and a representative of $h$ so that the following cyclic construction is well-defined: $F$ sends $\Sigma_i$ to $\Sigma_{i+1}$ for $1\leq i<m$, sends $\Sigma_m$ to $\Sigma_1$ by $h$, and acts on the central genus-zero component by the corresponding cyclic permutation of its boundary components. Thus $F^m|_{\Sigma_1}$ is $h$, up to the chosen identifications.

The central genus-zero component contributes no first homology to $S$. Hence
\[
H_1(S;\Q)\cong \bigoplus_{i=1}^m H_1(\Sigma_i;\Q),
\]
and, with respect to this splitting, $F_*$ has block-cyclic form. Applying formula~\eqref{formula:char_pol_for_block_cyclic_matrix} gives
\[
\chi_F(x)=\chi_{(F^m|_{\Sigma_1})_*}(x^m)=\chi_{h_*}(x^m)=q(x^m).
\]
For $m\geq 2$ the multicurve separating the positive-genus pieces from the central genus-zero component, or the core curve of the central annulus when $m=2$, is invariant, so $F$ is reducible. If $h$ is periodic, then the first-return maps on all complementary components are periodic; hence $F$ is of algebraically finite type.

With respect to the decomposition $H_1(S;\Q)\cong \bigoplus_{i=1}^m H_1(\Sigma_i;\Q)$, the map $F_*$ is represented by the block companion matrix
$$
F_* = 
\begin{pmatrix}
	0 & 0 & \cdots & 0 & h_*\\
	I & 0 & \cdots & 0 & 0\\
	0 & I & \cdots & 0 & 0\\
	\vdots & & \ddots & & \vdots\\
	0 & 0 & \cdots & I & 0
\end{pmatrix}.
$$
Thus $F_*^m=\operatorname{diag}(h_*,\ldots,h_*)$, so $F_*$ has order $mn$, since $h_*$ has order $n$.
\end{proof}

\begin{theorem}\label{theorem:classification_of_cyclotomic_pol_for_algebraically_finite_type}
Let $n\geq 3$. Then:
\begin{itemize}
	\item[(1)] $\varphi_n(x)$ is realizable as $\chi_f(x)$ for $f\in \Mod(S)$ of algebraically finite type if and only if $n$ has at most two distinct prime factors. Moreover, $f$ can be taken so that $f_*$ is of order $n$.
	
	\item[(2)] If $\chi_f(x)=\varphi_n(x)$ for algebraically finite type $f\in\Mod(S_g)$ and $n$ is equal to $4$, to a prime $p$, or to a product $pq$ of two distinct primes, then $f$ is periodic.
	\item[(3)] If $n$ is square-free and has at least three distinct prime factors and $\chi_f(x) = \varphi_n(x)$ for $f\in \Mod(S)$, then $f$ is pseudo-Anosov.
\end{itemize}
\end{theorem}

Part~(1) is precisely Theorem~\ref{theorem:intro_cyclotomic_classification} from the introduction.

\begin{proof}
		Assume first that $n$ has at most two distinct prime factors. We realize the required cyclotomic polynomial by reducing to the basic periodic cases and then applying Lemma~\ref{lemma:cyclic_substitution_construction}.
		
		For an odd prime $p$, the periodic order-$p$ model on $S_{(p-1)/2}$ can be chosen with a fixed point; deleting an invariant disk about this point gives a periodic mapping $h_p$ of $S_{(p-1)/2,1}$ with $\chi_{h_p}(x)=\varphi_p(x)$. For $p=2$, we instead use a periodic order-$4$ mapping $h_4$ of $S_{1,1}$ with $\chi_{h_4}(x)=\varphi_4(x)$. Finally, if $p\neq q$ are distinct primes, the deck transformation in Example~\ref{example:periodic_of_order_pq} fixes the point lying over $\infty$; deleting an invariant disk about this point gives a periodic mapping $h_{pq}$ of a one-boundary-component surface with $\chi_{h_{pq}}(x)=\varphi_{pq}(x)$.
		
		Now write $n$ in one of the following forms. If $n=p^a$ with $p$ odd, then
		\[
		\varphi_n(x)=\varphi_p(x^{p^{a-1}}),
		\]
		so Lemma~\ref{lemma:cyclic_substitution_construction}, applied to $h_p$ with $m=p^{a-1}$, gives an algebraically finite type mapping class with characteristic polynomial $\varphi_{p^a}(x)$. If $n=2^a$ with $a\geq2$, then
		\[
		\varphi_n(x)=\varphi_4(x^{2^{a-2}}),
		\]
		and the same construction applied to $h_4$ gives the desired realization. If $n=p^a q^b$ for distinct primes $p$ and $q$, then
		\[
		\varphi_n(x)=\varphi_{pq}(x^{p^{a-1}q^{b-1}})=\varphi_{pq}(x^{n/(pq)}),
		\]
		and applying the same construction to $h_{pq}$ gives an algebraically finite type mapping class realizing $\varphi_n(x)$. This proves the positive direction of (1). Moreover, in all cases the construction in Lemma~\ref{lemma:cyclic_substitution_construction} gives $f_*$ of the desired order $n$.
		
		Now suppose that $n$ has at least three distinct prime factors. Choose three of them and call them $p,q,r$, and write $n=pqrd$. In the unique expression~\eqref{formula:cyclotomic_polynomial}, the numerator indices include
		\[
		n,\qquad pd=\frac{n}{qr},\qquad qd=\frac{n}{pr},\qquad rd=\frac{n}{pq}.
		\]
		Indeed, these four indices are distinct, and
		\[
		\mu(n/n)=\mu(1)=1,
		\qquad
		\mu\bigl(n/(pd)\bigr)=\mu(qr)=1,
		\]
		with the same computation giving $\mu(n/(qd))=\mu(pr)=1$ and $\mu(n/(rd))=\mu(pq)=1$. Thus $\varphi_n(x)$ has at least four numerator indices in its unique rational product expression by factors of the form $x^k-1$.
		
		If, on the other hand, $\varphi_n(x)=\chi_f(x)$ for some algebraically finite type mapping class $f$, then $\chi_f(x)$ is irreducible over $\Z$, since $\varphi_n(x)$ is irreducible. Lemma~\ref{lemma:aft_irreducible_at_most_three_numerators} would then imply that the same unique rational product expression has at most three numerator indices. This contradiction proves the negative direction of (1).

		For (2), suppose that $f\in\Mod(S)$ is of algebraically finite type, $\chi_f(x)=\varphi_n(x)$, and $n$ is equal to $4$, to a prime $p$, or to a product $pq$ of two distinct primes. If $f$ is reducible, then by Lemma~\ref{lemma:reducible_mapping_and_irreducible_polynomial_r(x^k)} $\varphi_n(x) = r(x^k)$ for a symplectic polynomial $r(x)$ and $k>1$. In fact, $r(x)$ is a cyclotomic polynomial and so $n\neq \rad(n)$, which implies $n=4$. In that case, $r(x) = \varphi_2(x)$ which is not symplectic, a contradiction. Therefore $f$ is not reducible, so periodic.
		
		Finally, suppose $n$ is square-free and has at least three distinct prime factors and $\chi_f(x) = \varphi_n(x)$ for $f\in \Mod(S)$. By Lemma~\ref{lemma:reducible_mapping_and_irreducible_polynomial_r(x^k)} and the subsequent remark, $f$ is not reducible since $\chi_f(x) = \varphi_n(x)$ is not of the form $r(x^k)$ for $k>1$. By (1), $f$ is not of algebraically finite type, so it is not periodic. Therefore $f$ is pseudo-Anosov, which proves (3).
\end{proof}

\begin{example}
	For $g=3$, the values $n \in \{7,9,14,18\}$ are such that $6=\varphi(n)$, where $\varphi$ denotes Euler's totient function. By Theorem~\ref{theorem:classification_of_cyclotomic_pol_for_algebraically_finite_type}, all the corresponding $\varphi_n(x)$ are realized by algebraically finite type mapping classes.
\end{example}

\begin{corollary}\label{corollary:not_surjective_infinitely_many_genera}
	There are infinitely many genera $g$ for which there exists a polynomial $\varphi_n(x)$ with $2g=\varphi(n)$ that is not realizable by a mapping class of algebraically finite type; in particular, the corresponding characteristic polynomial is not realizable by a Morse--Smale diffeomorphism of $S_g$.
	
	Thus, for infinitely many such $g$, the restricted map
	\[
	\Psi|_{\AFT(S_g)}\colon \AFT(S_g)\to \QU(2g,\Z)\cap \Sp(2g,\Z)
	\]
	is not even surjective at the level of characteristic polynomials.
\end{corollary}

\begin{proof}
Let $p>3$ be prime and set $n=6p$. On the one hand, $n$ has three distinct prime divisors, so Theorem~\ref{theorem:classification_of_cyclotomic_pol_for_algebraically_finite_type} implies that $\varphi_n(x)$ is not realized by an algebraically finite type mapping class. On the other hand,
\[
\deg \varphi_{6p}=\varphi(6p)=2(p-1),
\]
so $\varphi_{6p}$ has the correct degree to be the characteristic polynomial of an element of $\Sp(2(p-1),\Z)$. Since there are infinitely many primes $p>3$, this gives infinitely many genera of the form $g=p-1$. The statement for Morse--Smale diffeomorphisms follows from da Rocha's characterization \cite{Rocha}, which says that a mapping class contains a Morse--Smale diffeomorphism if and only if it is of algebraically finite type.
\end{proof}

\begin{example}
	The number $n=2\cdot 3\cdot 5=30$ is the least such that $\varphi_n(x)$ is not realizable as $\chi_f(x)$ for $f\in \Mod(S)$ of algebraically finite type, and in fact is only realized by pseudo-Anosov mapping classes by Theorem~\ref{theorem:classification_of_cyclotomic_pol_for_algebraically_finite_type}. The corresponding surface has genus $g= \varphi(30)/2 = 4$, so this is the first possible genus in which the restriction of $\Psi$ to algebraically finite type mapping classes can fail to be surjective at the level of characteristic polynomials.
\end{example}

\begin{corollary}\label{corollary:realization_of_product_of_cyclotomic_polynomials}
	Let 
	$$
	p(x) = (x-1)^{2k_1} (x+1)^{2k_2}\prod_{i=3}^k \varphi_{n_i}(x)
	$$
	be a product of cyclotomic polynomials such that $k_1,k_2\geq 0$ and every $n_i\geq 3$ has at most two distinct prime factors. Then $p(x)$ is realizable as $\chi_f(x)$ for $f\in \Mod(S)$ of algebraically finite type such that $f_*$ is of order
	$$
	\begin{cases}
		\lcm\left(n_3,\ldots,n_k\right) \qquad \ \ \text{ if } k_2=0, \\
		\lcm\left({2,n_3,\ldots,n_k}\right) \qquad \text{otherwise.}
	\end{cases}
	$$
\end{corollary}

\begin{proof}
	The polynomials $(x-1)^{2k_1}$ and $(x+1)^{2k_2}$ are realized by the identity map $f_1$ on $\Sigma_1 = S_{k_1}$ and hyperelliptic involution $f_2$ on $\Sigma_2 = S_{k_2}$, respectively. By Theorem~\ref{theorem:classification_of_cyclotomic_pol_for_algebraically_finite_type} each polynomial $\varphi_{n_i}$, $i\geq3$, is realizable by some $f_i \in \Mod(\Sigma_i)$ of algebraically finite type such that $(f_i)_*$ is of order $n_i$. Moreover, the construction of each $f_i$ ensures that it has at least one fixed point.
	
	For every $i$ remove an invariant disc neighbourhood $D_i \subset \Sigma_i$ of a fixed point of $f_i$, resulting in a surface $\Sigma'_i$. Next, attach the boundary components of the surfaces $\Sigma'_i$ to the boundary components of a genus-zero surface $S_{0,k}$ along cylinders $Z_i \cong S^1\times [-1,1]$, obtaining a closed surface $S$. We define a mapping class $f$ of $S$ by taking $f_i|_{\Sigma'_i}$ on $\Sigma'_i$, the identity map on $S_{0,k}$, and on each cylinder $Z_i$ an isotopy between the actions on its end circles.
	
	The class $f$ is of algebraically finite type. For this it is sufficient to take a reduction system $C$ consisting of reduction systems for the maps $f_i$ on $\Sigma'_i$ making them algebraically finite, together with the core curves $S^1 \times \{0\}$ of the cylinders $Z_i$.
	
	Finally, $H_1(S;\Q)\cong \bigoplus_{i=1}^k H_1(\Sigma'_i;\Q) \cong \bigoplus_{i=1}^k H_1(\Sigma_i;\Q)$ because $S_{0,k}$ does not contribute to the first homology of $S$ and the inclusions $\Sigma'_i \hookrightarrow \Sigma_i$ induce isomorphisms on first homology groups. Since $f$ restricted to $\Sigma'_i$ is $f_i$ with the desired polynomial $\chi_{f_i}(x)$, we get $\chi_f(x) = p(x)$. It is clear that the order of $f_*$ is the least common multiple of orders of $(f_i)_*$.
\end{proof}


\section{A complete Casson--Bleiler-type criterion}

Recall the following criterion due to Casson--Bleiler (see \cite[Lemma 5.1]{Casson-Bleiler}), in the symplectically irreducible form used here (cf. \cite[Theorem 14.5]{Farb-Marg} and \cite{Margalit-Spallone}).

\begin{theorem}
	Let $f\in \Mod(S)$ and suppose that $\chi_f(x)$ satisfies each of the following conditions:
	\begin{itemize}
		\item[a)] it is symplectically irreducible over $\Z$,
		\item[b)] it is not a polynomial in $x^k$ for any $k > 1$;
		\item[c)] it is not a cyclotomic polynomial.
	\end{itemize}
	Then $f$ is pseudo-Anosov.
\end{theorem}

Theorem~\ref{theorem:classification_of_cyclotomic_pol_for_algebraically_finite_type}(3) can be viewed as a cyclotomic analogue of the Casson--Bleiler criterion. 
In fact, both these criteria state that under the conditions a symplectic polynomial is \emph{only} realizable as $\chi_f(x)$ for pseudo-Anosov mapping classes. This observation leads to a complete characterization of such polynomials, which restates Theorem~\ref{theorem:intro_complete_criterion} from the introduction.

%

We will use the following standard realization fact in two closely related forms. Every symplectic polynomial of degree $2g$ is the characteristic polynomial of some element of $\Sp(2g,\Z)$, and hence, by the surjectivity of $\Psi$, of some mapping class in $\Mod(S_g)$; see, for instance, \cite{Margalit-Spallone,Yang}. We will also use the corresponding one-boundary-component version: every symplectic polynomial of degree $2g$ is realized by the homological action of some mapping class of $S_{g,1}$. Indeed, the capping map $\Mod(S_{g,1},\partial S_{g,1})\to \Mod(S_g)$ followed by the symplectic representation is still surjective, since Dehn twists about nonseparating curves in $S_{g,1}$ map to the elementary symplectic transvections, and these generate $\Sp(2g,\Z)$.

\begin{theorem}\label{theorem:complete_pseudo_anosov_criterion}
	Let $p(x)\in\Z[x]$ be a symplectic polynomial of degree $2g$. Then
	$$
	\text{ $p(x)$ is realizable as $\chi_f(x)$ only for pseudo-Anosov mapping classes $f\in\Mod(S_g)$ }
	$$
	if and only if $p(x)$ satisfies each of the following conditions:
	\begin{itemize}
		\item[(a)] it is symplectically irreducible over $\Z$,
		\item[(b)] it is not a polynomial in $x^k$ for any $k > 1$;
		\item[(c)] it is not a cyclotomic polynomial, or $p(x) = \varphi_n(x)$ where $n$ has at least three distinct prime factors.
	\end{itemize}
\end{theorem}

Observe that if $p(x)=\varphi_n(x)$, then $(b)$ implies that $n$ is square-free.

\begin{proof}
First suppose that $p(x)$ satisfies $(a)$--$(c)$ and that $f\in\Mod(S_g)$ has $\chi_f(x)=p(x)$. If $p(x)$ is not cyclotomic, then the Casson--Bleiler criterion applies and $f$ is pseudo-Anosov. If $p(x)=\varphi_n(x)$ is cyclotomic, then $(b)$ implies that $n$ is square-free, and $(c)$ implies that $n$ has at least three distinct prime divisors. Theorem~\ref{theorem:classification_of_cyclotomic_pol_for_algebraically_finite_type}(3) then implies that $f$ is pseudo-Anosov.

It remains to show the converse. Thus suppose that at least one of $(a)$--$(c)$ fails. We show in each case that $p(x)$ is realized by a mapping class which is not pseudo-Anosov.

First assume that $(a)$ fails. Then $p(x)$ is symplectically reducible, so
\[
p(x)=p_1(x)p_2(x)
\]
for symplectic polynomials $p_i(x)$ of degrees $2g_i$, with $g_1+g_2=g$. Cut $S_g$ along a separating curve, obtaining two surfaces $S_{g_1,1}$ and $S_{g_2,1}$. By the one-boundary-component realization fact above, choose mapping classes $f_i\in\Mod(S_{g_i,1},\partial S_{g_i,1})$ such that $\chi_{f_i}(x)=p_i(x)$. Gluing the two surfaces along their boundary components and using $f_1$ and $f_2$ on the two sides gives a reducible mapping class $f\in\Mod(S_g)$ with
\[
\chi_f(x)=\chi_{f_1}(x)\chi_{f_2}(x)=p_1(x)p_2(x)=p(x).
\]
Hence $p(x)$ is not realized only by pseudo-Anosov mapping classes.

Next assume that $(c)$ fails. Then $p(x)=\varphi_n(x)$ for some $n$ with at most two distinct prime factors. By Theorem~\ref{theorem:classification_of_cyclotomic_pol_for_algebraically_finite_type}, this polynomial is realized by a mapping class of algebraically finite type. Such a mapping class is periodic or reducible and has no pseudo-Anosov component, so it is not pseudo-Anosov.

It remains to consider the case in which $(a)$ and $(c)$ hold but $(b)$ fails. Write
\[
p(x)=r(x^k)
\]
for some $k>1$. Since $p(x)$ is reciprocal, $r(x)$ is reciprocal as a polynomial in its own variable. We claim that $r(x)$ has even degree in the present case. If $\deg r$ were odd, then $k$ would be even, because $\deg p=2g$. Every reciprocal polynomial of odd degree has $-1$ as a root, so $r(y)=(y+1)s(y)$ with $s(y)$ reciprocal. Hence
\[
p(x)=(x^k+1)s(x^k).
\]
If $s$ is nonconstant, this is a factorization of $p(x)$ into two nonconstant symplectic polynomials, contradicting $(a)$. If $s=1$, then $p(x)=x^k+1$; symplectic irreducibility then forces $x^k+1$ to have a single cyclotomic factor, so $k$ is a power of two and $p(x)=\varphi_{2k}(x)$, one of the cyclotomic exceptions excluded by $(c)$. Thus $\deg r$ is even. Write $\deg r=2g'$, so that $g=kg'$. Then $r(x)$ is a symplectic polynomial. Choose a mapping class $h$ of $S_{g',1}$ with $\chi_h(x)=r(x)$, using the one-boundary-component realization fact. Applying Lemma~\ref{lemma:cyclic_substitution_construction} with $m=k$ gives a reducible mapping class $F\in\Mod(S_g)$ satisfying
\[
\chi_F(x)=r(x^k)=p(x).
\]
Again $p(x)$ is realized by a non-pseudo-Anosov mapping class. This proves the converse.
\end{proof}

\section{Maximal finite order in \texorpdfstring{$\Psi(\AFT(S_g))$}{Psi(AFT(Sg))}}

A classical theorem of Wiman from 1895, together with the Nielsen realization theorem, implies that the maximal order of a periodic mapping class on a closed orientable surface $S_g$ is $4g+2$ (see~\cite[Theorems~7.1 and~7.5]{Farb-Marg}). Although non-periodic mapping classes of algebraically finite type have infinite order by definition, it is still natural to ask how large the finite order of the symplectic transformation induced on first homology can be.

B\"{u}rgisser \cite[Corollary 2]{Burgisser} gave the following simple characterization of the finite orders of elements in $Sp(2g,\Z)$. Let $d= \prod_{i=1}^k p_i^{a_i}$ be the prime factorization of $d$, $p_i < p_{i+1}$ and $a_i \geq 1$. Define
$$
w(d) := \begin{cases}
	\sum_{i=2}^k \varphi(p_i^{a_i}) & \text{if } p_1^{a_1}=2, \\
	\sum_{i=1}^k \varphi(p_i^{a_i}) & \text{otherwise.}
\end{cases}
$$
Then $d$ is the order of an element of $\Sp(2g,\Z)$ if and only if 
$$
w(d) \leq 2g.
$$

Let $\mathcal A(g)$ denote the maximal finite order of a matrix in  $\Psi(\AFT(S_g)) \subset \Sp(2g,\Z)$.

\begin{theorem}\label{theorem:realization_of_orders_in_AFT_and_A(g)}
	If $d$ is the order of an element of $\Sp(2g,\Z)$, then there exists $f\in \Mod(S_g)$ of algebraically finite type such that the order of $f_*$ is equal to $d$. In particular, for every $g$ the number $\mathcal A(g)$ is equal to the maximal finite order of an element in $\Sp(2g,\Z)$.
\end{theorem}

\begin{proof}
	Let $d = \prod_{i=1}^k p_i^{a_i}$ be the prime factorization as above. The case $d=2$ is covered by the hyperelliptic involution of $S_g$. Assume $d\geq 3$. If $p_1^{a_1}=2$, then $k\geq 2$. Consider the polynomial
	\begin{align*}
		p(x) = (x-1)^{2g-w(d)} \cdot \begin{cases}
			\varphi_{2p_2^{a_2}}(x) \displaystyle\prod_{i=3}^k \varphi_{p_i^{a_i}}(x) \ &\text{ if } \ p_1^{a_1} =2 \\
			\displaystyle\prod_{i=1}^k \varphi_{p_i^{a_i}}(x) \ &\text{ otherwise.}
		\end{cases}
	\end{align*}
	Since the degree of $\varphi_{2p_2^{a_2}}(x)$ is $\varphi(p_2^{a_2})$, the degree of $p(x)$ is $2g$ and the exponent $2g-w(d)$ is even. Hence by Corollary~\ref{corollary:realization_of_product_of_cyclotomic_polynomials} there exists $f\in \Mod(S_g)$ such that $\chi_f(x) = p(x)$ and $f_*$ has finite order~$d$.	
\end{proof}

\begin{remark}
Let us comment on the asymptotic behavior of $\mathcal A(g)$. In the exposition paper \cite{Kuzmanovich-Pavlichenkov} of Kuzmanovich and Pavlichenkov they study, among other things, the asymptotic growth rate of the sequence $(\mathcal H(n))_n$, where $\mathcal H(n)$ is the maximal finite order of a matrix in $GL(n,\Z)$. Since finite orders of elements are the same in $\Sp(2g,\Z)$ and $GL(2g,\Z)$ by the result of B\"{u}rgisser \cite[Corollary~1]{Burgisser}, Theorem~\ref{theorem:realization_of_orders_in_AFT_and_A(g)} implies $\mathcal A(g) = \mathcal H(2g)$. It turns out, by a theorem of Levitt and Nicolas \cite{Levitt-Nicolas} also discussed in \cite[Theorem 3.10]{Kuzmanovich-Pavlichenkov}, that the logarithm of $\mathcal H(n)$ has the same asymptotic growth as the logarithm of Landau's function $\mathcal L(n)$ defined as the largest order of an element in the symmetric group of degree $n$. Namely,
$$
\lim_{n\to \infty} \frac{\ln \mathcal L(n)} {\sqrt{n\ln n}} =1
\qquad \text{ and } \qquad
\lim_{n\to \infty} \frac{\ln \mathcal H(n)} {\sqrt{n\ln n}} =1.
$$
This implies 
$$
\lim_{g\to \infty} \frac{\ln \mathcal A(g)} {\sqrt{2g\ln g}} =1,
$$
which can be stated as
$$
\mathcal A(g)=\exp\left((1+o(1))\sqrt{2g\ln g}\right),
$$
where $o$ is the little-o notation, that is, $o(1)$ denotes a function tending to zero as $g\to\infty$.

Thus, in contrast to the linear bound for the orders of periodic mappings, $\mathcal A(g)$ grows faster than any polynomial in $g$ and is exponential in $\sqrt{2g\ln g}$.
\end{remark}

\begin{remark}
	In \cite{GMMM} Graff et al. realize some invariants of periodic points related to Lefschetz numbers by mapping classes of algebraically finite type, motivated by the study of the dynamics of Morse--Smale diffeomorphisms. Instead of irreducible cyclotomic factors, they use polynomials of the form $(x^n-1)^2$ and permutation matrices for cycles as building blocks, see \cite[Theorem~3.1~and~4.7]{GMMM}. In this way, one can realize any order of an element in the symmetric group of degree $g$ as the order of an element in $\Psi(\AFT(S_g)) \subset \Sp(2g,\Z)$. In particular, the maximal possible order obtained using this approach is equal to the value $\mathcal L(g)$ of Landau's function, not to $\mathcal H(2g) = \mathcal A(g)$. 
\end{remark}

\section*{Acknowledgements}
The authors thank Dan Margalit for several helpful comments.


\end{document}